\documentclass[a4paper,12pt]{amsart} 
 \usepackage[T1]{fontenc}            
 \usepackage{amscd}                  
 \usepackage{xypic}                  
 \usepackage{amssymb}
\newtheorem{theorem}{Theorem}[section]

\title{The sequence of middle divisors is unbounded}
\author{Jon Eivind Vatne}
\begin{document}
\begin{abstract}
The sequence of middle divisors is shown to be unbounded.
For a given number $n$, $a_{n,0}$ is the number of divisors of $n$ in between $\sqrt{n/2}$ and $\sqrt{2n}$.
We explicitly construct a sequence of numbers $n(i)$ and a list of divisors in the interesting range, so that the length of the list goes to infinity as $i$ increases.
\end{abstract}

\maketitle
\section{Introduction}
In \cite{KR}, Kassel and Reutenauer studies the zeta function of the Hilbert scheme of $n$ points in the two-torus.
The polynomial counting ideals of codimension $n$ in the Laurent algebra in two variables turns out to have an interesting quotient, whose middle coefficient $a_{n,0}$ has a direct description:
\begin{equation*}
a_{n,0} =\left|\{d\,:\, d|n\,,\,\frac{\sqrt{2n}}{2}<d\leq \sqrt{2n}\}\right|.
\end{equation*}
We follow the symbolism from \cite{KR}, which the reader should also consult for more motivation.
In a talk at the conference {\em Algebraic geometry and Mathematical Physics 2016}, in honour of A. Laudal's 80th birthday, Kassel discussed the results in \cite{KR} and asked whether the sequence $a_{n,0}$ is bounded or not.
Evidently it grows very slowly.
The sequence is included in the online encyclopedia of integer sequences as sequence A067742 \cite{oeis}.\\

In this short note, we will show that the sequence is unbounded.
The idea is to choose $n$ such that $\sqrt{n/2}$ is a divisor, and to multiply this divisor with a number slightly larger than one repeatedly, making sure that the product still divides $n$ as long as it is smaller than $\sqrt{2n}$.

\section{Unboundedness of the sequence}
\begin{theorem}
Let 
\begin{equation*}
a_{n,0} =\left|\{d\,:\, d|n\,,\,\frac{\sqrt{2n}}{2}<d\leq \sqrt{2n}\}\right|.
\end{equation*}
Then
\begin{equation*}
\limsup_{n\to \infty} a_{n,0} = \infty
\end{equation*}
More precisely, for any $i\geq 1$ define $s_{max}=\ln(2)/\ln(1+i^{-1})$ and 
\begin{equation}
\label{Choicen}
n(i)= 2(i+1)^{\lceil 2s_{ max}\rceil}\cdot i^{2\lceil s_{max}\rceil}.
\end{equation}
Then $\lim_{i\to\infty}a_{n(i),0} =\infty$.
\end{theorem}
\begin{proof}
With the choice of $n(i)$ from \eqref{Choicen}, we have that
\begin{equation*}
\sqrt{n/2} = (i+1)^{\lceil s_{ max}\rceil}\cdot i^{\lceil s_{max}\rceil},
\end{equation*}
a divisor of $n(i)$.
For each $s=1,2,\dots,\lfloor s_{max}\rfloor$, consider
\begin{equation*}
d(s) = \sqrt{n/2}\left(\frac{i+1}{i}\right)^s = (i+1)^{\lceil s_{ max}\rceil+s}\cdot i^{\lceil s_{ max}\rceil-s}.
\end{equation*}
This divides $n(i)$ as long as $\lceil s_{ max}\rceil+s\leq 2 \lceil s_{ max}\rceil$ and $\lceil s_{ max}\rceil-s\geq 0$, which in both cases translates simply to $s\leq \lfloor s_{max}\rfloor$.
Thus we have exhibited a number of divisors, so that
\begin{equation*}
a_{n(i),0} \geq \lfloor s_{max}\rfloor.
\end{equation*}
Note also that $s_{max}$ is chosen so that
\begin{equation*}
\left(\frac{i+1}{i}\right)^{s_{max}} = 2.
\end{equation*}
Therefore all the $d(s)$ are in the interesting interval.
Since
\begin{equation*}
\lim_{i\to\infty}s_{max}(i) = \lim_{i\to\infty} \frac{\ln 2}{\ln(1+i^{-1})} = \infty
\end{equation*}
this proves the theorem.
\end{proof}

The sequence $n(i)$ grows very quickly whereas as the sequence $s_{max}(i)$ grows slowly.
It is likely that the minimal $n$ needed to find a given value for $a_{n,0}$ is a lot smaller than what is constructed in the proof.\\

{\bf Acknowledgements}\\
Thanks are due to C. Kassel for telling me about this problem and encouraging me to write down the proof.

{\small
{\em Authors' address}:
{\em Jon Eivind Vatne}, Department of Computing, Mathematics and Physics, Faculty of Engineering, Bergen University College, PO.Box 7030, N-5020 Bergen, Norway ...
 e-mail: \texttt{jon.eivind.vatne@\allowbreak hib.no}.

}
\end{document}